\documentclass[letterpaper, 11pt,  reqno]{amsart}
\usepackage[margin=1.2in,marginparwidth=1.5cm, marginparsep=0.5cm]{geometry}

\usepackage{amsmath,amssymb,amscd,amsthm,amsxtra, esint, xcolor,mathtools,enumerate}

\usepackage[implicit=true]{hyperref}
\allowdisplaybreaks[2]

\usepackage{color}

\definecolor{gr}{rgb}   {0.,   0.69,   0.23 }
\definecolor{bl}{rgb}   {0.,   0.5,   1. }
\definecolor{mg}{rgb}   {0.85,  0.,    0.85}
\definecolor{yl}{rgb}   {0.8,  0.7,   0.}
\definecolor{or}{rgb}  {0.7,0.2,0.2}

\newtheorem{theorem}{Theorem} [section]

\newtheorem{proposition}[theorem]{Proposition}
\newtheorem{remark}[theorem]{Remark}

\newtheorem{definition}[theorem]{Definition}

\newcommand{\II}{\text{I \hspace{-2.8mm} I} }

\newcommand{\noi}{\noindent}

\newcommand{\R}{\mathbb{R}}
\newcommand{\C}{\mathbb{C}}

\newcommand{\bul}{\bullet}

\newcommand{\E}{\mathbb{E}}

\newcommand{\DD}{\mathbb{D}}

\newcommand{\K}{\mathbf{K}}

\newcommand{\F}{\mathcal{F}}

\newcommand{\al}{\alpha}
\newcommand{\be}{\beta}
\newcommand{\dl}{\delta}

\newcommand{\eps}{\varepsilon}

\newcommand{\g}{\gamma}

\newcommand{\s}{\sigma}

\newcommand{\dt}{\partial_t}

\newcommand{\LRA}{\Longrightarrow}

 \newcommand{\w}{\textup{w}}

\renewcommand{\o}{\omega}
\renewcommand{\O}{\Omega}

\newcommand{\les}{\lesssim}

\newcommand{\jb}[1]
{\langle #1 \rangle}

\newcommand{\ind}{\mathbf 1}

\newcommand{\PP}{\mathbb{P}}

\newcommand{\N}{\mathbb{N}}

\renewcommand{\H}{\mathcal{H}}

\newcommand{\NN}{\mathcal{N}}
\newcommand{\cA}{\mathcal{A}}

\newcommand{\Var}{\textup{Var}}

\newcommand{\wt}{\widetilde}
\newcommand{\wh}{\widehat}

\numberwithin{equation}{section}
\numberwithin{theorem}{section}

\makeatletter
\@namedef{subjclassname@2020}{\textup{2020} Mathematics Subject Classification}
\makeatother

\begin{document}
\baselineskip = 14pt

\title[ASCLT for  PAM/HAM with colored noises]
{Almost sure central limit theorems for      parabolic/hyperbolic   Anderson models  \\ with  Gaussian colored noises}

\author[P. Xia  and G. Zheng]
{Panqiu Xia and Guangqu Zheng}

\thanks{Panqiu Xia is partially supported by NSF grant DMS-2246850.}

\address{
Panqiu Xia\\
School of Mathematics\\
Cardiff University\\
Abacws,  Senghennydd Road\\
Cathays, Cardiff\\
CF24 4AG, United Kingdom}

\email{xiap@cardiff.ac.uk}

\address{
Guangqu Zheng,
Department of Mathematics and Statistics\\
Boston University\\
665 Commonwealth Avenue\\
Boston, MA 02215, USA
}

\email{gzheng90@bu.edu}



\begin{abstract}  
This short note is devoted to 
establishing the almost sure central limit theorem 
for the parabolic/hyperbolic Anderson models driven by 
colored-in-time Gaussian noises,
completing recent results on quantitative central limit theorems 
for stochastic partial differential equations. 
We combine the second-order Gaussian Poincar\'e inequality with
Ibragimov and Lifshits' method of characteristic functions, 
effectively overcoming the challenge 
 from the lack of It\^o tools in this colored-in-time 
setting, and achieving results that are inaccessible with previous methods.

\end{abstract}

\date{\today}
\maketitle
%

%

\baselineskip = 14pt

\section{Introduction}

The classical central limit theorem (CLT) states that
for a random sample of size $n$  drawn from a population with mean zero
and variance one, the sample mean $M_n$ admits Gaussian fluctuation 
as the sample size $n$ tends to infinity:
\[
\sqrt{n} M_n \xrightarrow[n\to\infty]{\rm law} \NN(0,1).
\]

\noi
Following this setting, the   almost sure central limit theorem (ASCLT)  
in its simplest form asserts that 
one can observe a Gaussian  behavior (asymptotically) 
along a generic trajectory via a logarithmic 
average: 
for almost every $\o\in\O$

\noi
\begin{align}\label{ASCLT1}
\frac{1}{\log n} \sum_{k=1}^n \frac{1}{k} \dl_{\sqrt{n} M_n(\o)  } \LRA \zeta
\end{align}

\noi
as $n\to\infty$, where $\dl_x$ denotes the Dirac mass at $x\in\R$, 
``$\LRA$'' indicates the weak convergence of finite measures, 
and $\zeta\sim\NN(0,1)$ stands for the standard Gaussian measure on $\R$
{\it throughout this note}. 
The first  ASCLT was   introduced  by P. L\'evy in his book
\cite[page 270]{Levy54} but remained largely unnoticed for several decades,
until it was rediscovered by various researchers in probability 
and dynamical systems
\cite{BD87, Fisher87, Bro88, Schatte88, LP90, Vol90, AW96}. 
For a comprehensive historical account up to 2001, see also the work
of Berkes and Cs\'aki \cite{BC01}.
 In recent years, several works
 \cite{BNT, CZthesis, Zheng17, AN22} have established almost  sure 
 (non-)central limit theorems using variants of Malliavin calculus
 in Gaussian, Poisson, and Rademacher settings. 

 On a different note, around 2018, Huang, Nualart, and Viitasaari initiated in \cite{HNV} a study on (quantitative) 
 central limit theorems for stochastic partial differential equations (SPDEs) driven by Gaussian noises.
 More precisely, they established a quantitative  CLT
 for the spatial averages of the solution to  
 a stochastic nonlinear heat equation 
 with multiplicative Gaussian space-time white noise. 
 Additionally, the first attempt of the similar topic for  
 stochastic nonlinear wave equations was published in 
 \cite{DNZ}.
 Since this pioneering work \cite{HNV}, 
 there has been a rapidly growing literature
 on the spatial averages of SPDEs. For CLT results, see, e.g., 
\cite{HNVZ,NZ20a,  BNZ, BNQSZ, NXZ22, CKNP22, CKNPjfa}, and for spatial ergodicity that precedes the CLTs, refer to \cite{CKNP21, NZ20b}.
 We refer the interested  readers 
 to \cite[(incomplete) table on page 5]{BHWXY23} 
 for an overview of relevant results in the Gaussian setting 
 and \cite{BZ23} in the L\'evy setting.

 In a series of papers  \cite{LZ1, LZ2, LZ3},   
 Li and Zhang developed  the ASCLTs for several SPDEs 
 with Gaussian noises that are white in time. 
 A key tool they employed is Malliavin calculus, 
 particularly the Clark-Ocone formula, which replies heavily on 
 the martingale structure resulting from the white-in-time nature of the Gaussian noises; 
  see  Remark \ref{rem_comp}-(ii) for more details.
  For a similar treatment applied 
  on the hyperbolic Anderson model driven by space-time pure-jump L\'evy white noise, refer to
  \cite[Section 3.1]{BXZ23}.  
However, this strategy using Clark-Ocone formula fails when attempting to establish  
the ASCLT results for cases with colored-in-time Gaussian noises as in, e.g., 
  \cite{BNQSZ, NXZ22}. Motivated by our joint work with Balan \cite{BXZ23}, 
  we will use a combination of  Ibragimov-Lifshits' method of characteristic functions
   and the second-order Gaussian Poincar\'e inequality to establish the ASCLT results
   for hyperbolic and parabolic Anderson models (HAM/PAM) driven by  
   space-time Gaussian colored  noises. 
   Such a combination was originally introduced in \cite[Section 3.2]{BXZ23} 
   within the L\'evy setting.

 \subsection{Framework}

Consider the following two equations:

\noi
\begin{align}\label{pam} \tag{PAM}
  \begin{cases}
  \dt u = \tfrac{1}{2} \Delta u + u \diamond \dot{W},\\
  u (0, \cdot) \equiv 1;
  \end{cases}
\end{align}

\noi
and

\noi
\begin{align}\label{HAM} \tag{HAM}
\begin{cases}
\dt^2 u = \Delta u + u \diamond \dot{W} \\
u(0,\cdot) \equiv 1 
\quad
{\rm and}
\quad
\dt u(0, \cdot) \equiv 0;
\end{cases}
\end{align}

\noi
where $\diamond$ denotes the Wick product, 
meaning that  the corresponding stochastic integral is interpreted 
in the Skorohod sense,
 and $\dot{W}$ is a centered space-time Gaussian noise with correlation
 
 \noi
\begin{align*} 
  \E \big[\dot{W} (t,x) \dot{W}(s,y)\big] = \gamma_0 (t - s) \gamma_1 (x-y)
\end{align*}

\noi
satisfying certain conditions to ensure the existence and uniqueness of 
(random field) solutions. In this note, we consider \eqref{pam} in any spatial dimension,
while
we restrict our analysis of \eqref{HAM} to dimensions $d = 1, 2$; see Remark \ref{rem_d} for
relevant discussions. 
Let us first state the following standing hypotheses of this note.

\noi
\begin{enumerate}[\bf (H1)]
\item $\gamma_0 : \R \to [0,\infty]$ is nonnegative-definite and locally integrable.

\item $\gamma_1: \R^d\to[0,\infty]$ is nonnegative-definite
 such that  $\g_1 = \F \mu$ is  the Fourier transform  
 of some nonnegative tempered measure $\mu$, called the spectral measure, 
 satisfying Dalang's condition (\cite{Dalang}):
  \begin{align}\label{D99}
  \int_{\R^d} \frac{1}{ 1  + | \xi|^2 } \mu(d\xi) <\infty.
  \end{align}
 
\end{enumerate}

\begin{definition} \rm
A random field  $u = \{u(t,x) : (t,x) \in \R_+\times\R^d\}$ is called a solution to    \eqref{pam} or   \eqref{HAM},  provided that for all $(t,x) \in \R_+ \times \R^d$,
it holds almost surely that 

\noi
\begin{align}\label{mild}
  u(t,x) = 1 + \int_0^t \int_{\R^d} G_{t - s} (x - y) u(s, y) W(d s, d y),
\end{align}
where $G = G^H$ (resp. $G^W$) denotes the heat kernel (resp. wave kernel)
 \begin{align}
G_{t}^H(x) \coloneqq (2 \pi t)^{- \frac{d}{2}} e^{- \frac{|x|^2}{2 t}} ,
 \,\,\,\text{ $\forall d\geq 1$},
\quad \text{and} \quad G_t^W(x) \coloneqq \begin{dcases}
\frac{1}{2} \ind_{\{ |x| < t \}}, & d = 1,\\
\frac{1}{2 \pi \sqrt{t^2 - |x|^2}}  \ind_{\{ |x| < t \}},   & d = 2;
\end{dcases}
\label{fSol}
\end{align}
and the stochastic integral is interpreted in the Skorohod sense.
That is, 
the mild form \eqref{mild} should be understood as 
$u(t,x) = 1 + \dl(  G_{t-\bul}(x-\ast) u(\bul, \ast)   )$ with 
$\dl$ the divergence operator in Malliavin calculus; 
see Section \ref{SEC2_M} and see also 
 \cite[Section 1.3.2]{Nua06} and \cite[Section 2.5]{blue}.
Note that we follow the convention that $G_t = 0$ for $t\leq 0$.

\end{definition}

Due to the linearity in the unknown $u$, 
one can formally iterate the  integral equation
\eqref{mild} to obtain an infinite series, 
which corresponds exactly to the {\it Wiener chaos expansion}
of the solution (whenever it exists):

 \noi
 \begin{align}\label{CP1}
  u(t,x) = 1 + \sum_{p = 1}^{\infty} I_p (f_{t,x,p}),
\end{align}

\noi
where $I_p (f_{t,x,p})$ stands for the $p$-th multiple Wiener-It\^{o} integral 
of the kernel $f_{t,x,p}$
 given by

\noi
\begin{align*}
  f_{t,x,p}(s_1,\dots, s_p, z_1,\dots, z_p) 
  \coloneqq \frac{1}{p!} \prod_{i = 0}^p G_{t_{\tau(i + 1)} 
        - t_{\tau (i)}} \big(z_{\tau(i + 1)} - z_{\tau (i)}\big),
\end{align*}

\noi
with $\tau$ denoting the permutation on $\{1,\dots, p\}$ 
such that $0 < t_{\tau (1)} < \dots < t_{\tau_{(p)}} < t$, 
and by convention $(t_{\tau (0)}, z_{\tau (0)})  \coloneqq (0,0)$ 
and $(t_{\tau (p + 1)}, z_{\tau (p + 1)}) \coloneqq (t,x)$;
see Section \ref{SEC_pre} for some preliminaries. 
The finiteness of the above series in $L^2(\O)$, 
or the validity of the chaos expansion \eqref{CP1},
can be verified by computing the second moment of each multiple integral, 
which  relies on the orthogonality relation 
(see, e.g., \cite[Proposition 2.7.4]{blue}).

\begin{remark}\rm
In fact, the above hypotheses {\bf (H1)} and {\bf (H2)} suffice to 
guarantee the unique existence of solutions to \eqref{pam}
and \eqref{HAM}; see  \cite[Theorem 3.2]{HHNT15} for the heat case
and see \cite[Section 5]{BS17} for the wave case. 
For the wave case, 
Balan and Song proved the unique existence of 
solutions to \eqref{HAM}
on any spatial dimension. More precisely, 
{\bf (H1)} and {\bf (H2)} implies that \eqref{CP1} is the unique solution 
to \eqref{HAM} when $d\leq 2$ (see \cite[Theorem 5.2]{BS17}), 
and under the additional 
assumption that the spectral measure $\mu$
 is absolutely continuous with 
respect to the Lebesgue measure, \eqref{CP1}
 is the unique solution 
to \eqref{HAM} when $d\geq 3$ 
(see \cite[Theorem 5.6]{BS17}. 
The delicacy in higher dimensions comes from 
the fact that the corresponding wave kernel $G^W$ on $\R^d$, 
for $d\geq 3$, 
is not  a function any more, so that the interpretation of 
product of  $G^W$ and the unknown $u$ 
(and thus the interpretation of 
the multiple integral $I_p(f_{t,x,p})$) requires additional care.
Such delicacy also makes it more difficult to establish
 the CLT results 
in higher dimensions (see \cite{Ebina1, Ebina2}).

\end{remark}

In the following, we present several results on (quantitative) CLTs
for spatial averages, which will serve as the basis for establishing the ASCLTs. 

\begin{theorem}\label{thm_old} 
Let  the above hypotheses {\bf (H1)} and {\bf (H2)} hold and 
we   make further assumptions on the temporal/spatial correlation kernels:
\begin{itemize}
\item[(a)] for any $\eps > 0$,
\begin{align}
\label{cond_trivial}
 \int_0^\eps \gamma_0 (r )dr > 0;
\end{align}

\item[(b)] the spatial correlation kernel $\gamma_1$ satisfies
 one of  the following properties:
\begin{enumerate}
  \item[\rm (b1)] $\displaystyle \gamma_1 \in L^1(\R^d)$ with 
  $\| \gamma_1 \|_{L^1(\R^d)} > 0$,

  \item[\rm(b2)] $\gamma_1 (z) = |z|^{-\alpha}$  
  with some $\alpha \in (0, 2 \wedge d)$.

  \end{enumerate}
\end{itemize}
Let $u$ denote the solution \eqref{CP1} to \eqref{pam} 
for any spatial dimension $d$ or \eqref{HAM} with $d\leq 2$.
Fix any $t_0\in(0,\infty)$ throughout this note. 
We define  for any $R > 0$ that

\noi
 \begin{align}\label{def_f}
 F_R &= F_R (t_0)
   \coloneqq \int_{|x| \leq R} [ u(t_0, x) - 1 ] dx
   \quad
   {\rm and}
   \quad
\wh{F}_R \coloneqq \tfrac{1}{\s_R} F_R
 \end{align}
with  $\s_R \coloneqq \sqrt{ \Var(F_R) } > 0$ for each $R>0$. 
 Then, the following quantitative CLTs hold for any $R > 0$:
 
\begin{align} 
 d_{\rm TV}(\wh{F}_R, Z ) \leq C  \times
 \begin{cases}
 \begin{aligned}
  R^{-d/2}    & \quad \text{\rm for \eqref{pam} with (b1)} &&  \text{\rm \textbf{(case 1)}}  \\
 R^{-\al/2} &  \quad   \text{\rm for \eqref{pam} with (b2)} &&     \text{\rm \textbf{(case 2)}}  \\
 R^{-d/2}   & \quad \text{\rm for \eqref{HAM} with (b1)} &&  \text{\rm \textbf{\textbf{(case 3)}}}  \\
   R^{-\al/2}      & \quad \text{\rm for \eqref{HAM} with (b2)} &&   \text{\rm \textbf{(case 4)}},
   \end{aligned}
 \end{cases}
 \label{QCLT1}
 \end{align}

 \noi
 where $Z \sim \mathcal{N}(0,1)$,
 $d_{\rm TV}(X, Y)$ stands for the total-variation distance between 
 two real-valued random variables $X$ and $Y$:

\noi
\begin{align}\label{TVD}
 d_{\rm TV}(X, Y ) 
 \coloneqq  \tfrac{1}{2}\sup \big| \E[ h(X)]  - \E[h(Y)] \big|,
\end{align}

\noi
where the supremum running over all real-valued bounded measurable functions
$h:\R\to\R$ such that $\| h \|_\infty \leq 1$. 
Here,  the implicit constant $C$ does not depend on $R$.
 See \cite{NXZ22, BNQSZ} for more details.

 \end{theorem}

The range of $\al$ in (b2)
  follows from \eqref{D99}, 
  with the corresponding
  spectral measure $\mu(d\xi)  = c_\al | \xi|^{\al - d}$
  for some explicit constant $c_\al$.
  The conditions \eqref{cond_trivial} and ``$\| \g_1\|_{L^1(\R^d)} > 0$"
  ensure that we are dealing with nontrivial space-time Gaussian noises
  and that the spatial integral $F_R$ has strictly positive variance 
  for each $R$. 
  It is 
not difficult to prove  $\s_R>0$ for each $R>0$.
 In fact, for $t_0>0$, $\E[ u(t_0, x)^2 ] > 1 $,
 in view of the chaos expansion \eqref{CP1} and the orthogonality relation of multiple integrals. The same argument also leads to $ \E[ u(t_0, x) u(t_0, y) ] \geq 1$, so that
 with these two facts   and the $L^2(\Omega)$-continuity of the solution, we get 
 \[
 \s_R^2 = \int_{|x|, |y| < R} \big( \E[ u(t_0, x) u(t_0, y) ] - 1 \big) dxdy > 0.
 \] 
 For the asymptotic behavior of $\s_R$, see \eqref{sR_asymp}.

 \subsection{Main result}

Let us first state the definition of ASCLT.

\begin{definition}\label{ASCLT}
 A family  $\{ F_{\theta} \colon \theta \geq 1\}$ of real random variables is said to satisfy
the  \textup{ASCLT} if for $\PP$-almost every
 $\omega\in\Omega$, the map $\theta \mapsto F_{\theta}(\omega)$
  is almost surely measurable and

 \noi
 \begin{align} \label{def2}
\nu_T^\o
\coloneqq \frac{1}{\log T} \int_1^T \dl_{F_{\theta}(\omega)}   \frac{d \theta}{\theta}
 \LRA \zeta
 \quad\text{as $T\to+\infty$,}
  \end{align}
where $\zeta$ stands for the standard Gaussian measure 
on $\R$. 

 \end{definition}

 In this note, we aim
 to establish an ASCLT for  $\wh{F}_R$  in \eqref{def_f}. 
 Thus, we choose  the continuum parameter $R>0$ in 
Definition \ref{ASCLT}, unlike the discrete parameter $n\in\N$
stated  in \eqref{ASCLT1}. 
 We present  an equivalent statement of ASCLT in 
Remark \ref{rem_comp}-(ii).

\begin{remark} \label{rem_comp}\rm

(i)
 The \emph{logarithmic average} in \eqref{ASCLT1} can be 
replaced by any slowly varying analogue. For example, $1/k$ and $\log n$ in \eqref{ASCLT1}, 
can be  substituted with  any $d_k$ and $D_n > 0$, respectively, 
provided that $D_n = d_1 + \dots + d_n\uparrow \infty$ 
and $D_{n+1}/D_n \to 1$ as $n$ tends to infinity. 
See also the discussion in \cite[page 204]{LP90}. 
Note that our main result  (Theorem \ref{thm_main})  remains valid  
when the logarithmic average is replaced by a slowly varying analogue
 in \eqref{def2}, 
although we retain the former for simplicity in presentation. 

\smallskip

\noi
(ii)
The validity of the ASCLT (Definition \ref{ASCLT})
for $\{ F_\theta: \theta \geq 1\}$
is equivalent to the statement that
for any bounded and Lipschitz continuous function $\phi:\R\to\R$,
it holds almost surely that
\[
\frac{1}{\log T} \int_1^T \frac{1}{\theta} \phi( F_\theta ) d\theta
\to \int_{\R} \phi(x) \zeta(dx),
\]
as $T\to+\infty$; see \cite[Remark 1.2]{BXZ23}.
Assuming CLT holds (i.e., $F_\theta\LRA \zeta$), it is equivalent to show
that, almost surely, 
\[ 
\frac{1}{\log T} \int_1^T \frac{1}{\theta} H_\theta   d\theta
\to \int_{\R} \phi(x) \zeta(dx),
\]
where $H_\theta \coloneqq  \phi( F_\theta ) - \E[  \phi( F_\theta ) ]$
is uniformly bounded. 
Taking advantage of the 
uniform boundedness of $H_\theta$,
it suffices to obtain a power decay
like   $|\E[ H_\theta H_\w]| 
\leq C (\theta/\w)^\beta$
for $\theta < \w$ with some $\be>0$ and $C\in(0,\infty)$. 
This can be easily accomplished 
using Clark-Ocone formula when the noise is white in time; see, e.g., 
\cite[Section 3.1]{BXZ23} and \cite{LZ1, LZ2, LZ3}.
However, in our note, where the Gaussian noises are colored in time, 
this strategy of   applying the Clark-Ocone formula fails. 
\end{remark}


Now we are ready to state our main result.

\begin{theorem}\label{thm_main}
Under the assumptions of Theorem \ref{thm_old}, the {\rm ASCLT}
holds for  $\big\{ \wh{F}_R: R\geq 1 \big\}$.

\end{theorem}

   \begin{remark}\label{rem_d}
   \rm
(i) In this note, we focus exclusively on the wave equation in low spatial dimensions
(i.e., $d\leq 2$). As previously mentioned, in higher dimension, 
the wave kernels $G^W$, unlike the those in \eqref{fSol}, 
are no longer functions, leading to extra difficulty in employing the
Malliavin-Stein method to establish the desired CLTs. 
Recent works by M. Ebina in \cite{Ebina1, Ebina2}, has successfully
established these CLTs for the stochastic nonlinear wave equation on $\R^d$
with $d\geq 3$.  It is then a natural extension to study the ASCLT 
in this high dimensional setting.

\smallskip
\noi
(ii)  
For the heat equation, we
simplify the presentation by focusing only on the regular case 
from \cite{NXZ22}, 
where the spatial correlation kernel is a nonnegative function. 
We do not address the rough case, which involves generalized functions 
for the spatial correlation kernel.  For instance,
in one-dimensional case, this includes the correlation 
whose spectral measure is $\mu(d\xi) = C_{H_1} |\xi|^{1-2H_1}$,
and $\gamma_0(t) = |t|^{2H_0 -2}$ for some explicit constant 
$C_{H_1}>0$
and $0 < H_1 < 1/2 < H_0 < 1$ such that
$H_0 + H_1 > 3/4$, or even more rough situations 
as discussed in \cite{LHW22}. 
We are optimistic that 
our strategy can be applicable in this rough case,
while  such an  investigation would inevitably require  
those very technical estimates from  \cite{NXZ22}.
   
   \end{remark}

Let us comment a bit our strategy and postpone the details 
to Section \ref{sec_proof}.
We will apply the powerful  Ibragimov-Lifshits' criterion
to prove Theorem \ref{thm_main}.

 \begin{proposition}\label{prop:IL} {\rm(}\cite[Ibragimov-Lifshits' criterion]{IL99}{\rm)}
A family of real-valued random variables $\{F_\theta\}_{\theta\geq 1}$ satisfies the {\rm ASCLT}
if $\theta\mapsto F_\theta$ is measurable almost surely, and
the following inequality holds

 \noi
\begin{align}
    \sup_{|s| \leq T}
    \int_2^\infty
    \dfrac{\E \big[  |\K_t(s)|^2\big]  }{t  \log t} dt <\infty,
  \label{IL}
\end{align}
for any finite $T >0$, where
\begin{align}\label{Kn}
\K_t(s)
\coloneqq \frac{1}{\log t} \int_1^t \frac{1}{\theta}
\big( e^{is F_\theta} - e^{-s^2/2} \big)  d\theta,
\quad t \in( 1, \infty).
\end{align}
\end{proposition}

In Ibragimov-Lifshits' original paper \cite{IL99},
the criterion
is proved for the discrete-time version.
For the proof of the above  continuum version,
see \cite[Proposition 3.3]{BXZ23}.
Note that Ibragimov-Lifshits' criterion
is not limited to the Gaussian limit (i.e., ASCLT),
as one can see from the original paper
and also from the recent application \cite{AN22}.
We expect that it will be useful in establishing
almost sure non-central limit theorem in the SPDE context,
for example, an almost sure non-central limit theorem in the framework of 
\cite{DT21}.

The said criterion requires us essentially to establish
logarithmic decay in  the second moment 
of the  difference of characteristic functions
of the (random)  probability $\nu_T^\o$ \eqref{def2}
 and standard Gaussian measure.
By some simple algebra, we only need to
bound the total-variation distances $d_{\rm TV}(  F_\theta, \NN(0,1))$
and  $d_{\rm TV}( \frac{   F_\theta -   F_\w}{\sqrt2}, \NN(0,1))$.
  While not claiming any originality,
we state below an abstract result that concludes  this discussion. 

\begin{proposition} \label{prop1818}
A family $\{F_\theta: \theta\geq 1\}$ satisfy the {\rm ASCLT} if 
\begin{align}\label{TV1a}
d_{\rm TV}(F_\theta,  \NN(0,1)) \leq C_1  \theta^{-\be_1}
\end{align}
 and
 \begin{align}\label{TV1b}
 d_{\rm TV} \big( \tfrac{F_\theta - F_\w}{\sqrt{2}},  \NN(0,1)\big) 
 \leq C_2  
 \big( \theta^{-\be_2}
 +  (\theta/\w)^{\be_3} \big) \,\,\text{for $\theta < \w$},
 \end{align}
 
 \noi
 where $C_1, C_2$ are constants that do not depend on $\theta$ and $\w$,
 whilst  $\be_i > 0$ for $i=1,2,3$.
 The above result remains valid, if the total-variation distances
 in \eqref{TV1a}-\eqref{TV1b} are replaced by $1$-Wasserstein distances.
\end{proposition}

Here, the $1$-Wasserstein distance $d_{\rm Wass}(X, Y)$
of two real-valued random variables $X, Y$ is defined by
\begin{align}\label{WassM}
d_{\rm Wass}(X, Y)
:= \sup_{\| h' \|_\infty \leq 1} \big| \E[ h(X)] - \E[ h(Y)]  \big|,
\end{align}

\noi
where the supremum in \eqref{WassM} runs over all $1$-Lipschitz functions $h:\R\to\R$.
In some contexts (e.g., \cite{BXZ23}), 
it is more natural to use the bounds in  $1$-Wasserstein distances
in place of  \eqref{TV1a}-\eqref{TV1b}.
We postpone the proof of Proposition \ref{prop1818} to Section \ref{SEC_pre}.

The rest of this note is organized as follows:
We collect a few preliminaries in Section \ref{SEC_pre}
and we present the proof of Theorem \ref{thm_main}
in Section \ref{sec_proof}.

\section{Preliminaries}\label{SEC_pre}

In Section \ref{SEC2_M}, we present a few basics on 
Malliavin calculus, while in Section \ref{SEC2_IL},
we record a criterion of 
 Ibragimov and Lifshits.

\medskip

Assume that all probabilistic objects in this note are defined
on a rich enough probability space $(\O, \F, \PP)$.
We write $\| \bul\|_p$ to denote the $L^p(\O)$-norm
for $p\in[1,\infty]$
and 
we write $a (R) \les b (R)$ for
 $ \limsup_{R\to+\infty }a(R)/ b(R) < +\infty$,
 and $a(R)\sim b(R)$
 for 
 \[
 0< \liminf_{R\to+\infty }a(R)/ b(R) \leq  
  \limsup_{R\to+\infty }a(R)/ b(R) < +\infty,
  \] 
for any nonnegative functions $a$ and $b$.

Suppose the Gaussian noise $\dot{W}$ is defined as in the Introduction
such that both hypotheses {\bf (H1)} and  {\bf (H2)} hold. 
Under these conditions, we can rigorously built the isonormal framework 
needed for developing
the $L^2$ theory of Malliavin calculus. Let $C_c(\R_+\times\R^d)$
denote the set of all real continuous functions on $\R_+\times\R^d$ with compact support. We define the following inner product on $C_c(\R_+ \times \R^d)$: 

\noi
\begin{align}\label{iso_inner}
\langle h_1, h_2 \rangle_\H 
\coloneqq \int_{\R_+\times\R^{2d}}
 h_1(r, y)  h_2(r', y') \g_0(r-r') \g_1(y-y') drdr' dydy'
\end{align}

\noi
for any $h_1, h_2\in C_c(\R_+\times\R^d)$. Let $\H$ denote the closure 
of $C_c(\R_+\times\R^d)$ under the above inner product \eqref{iso_inner}. A random field $W = \{W(h): h\in\H\}$ is an isonormal Gaussian process
over the Hilbert space $\H$, if 
$W$ is a centered Gaussian family with covariance given by
\[
\E[W(h_1) W(h_2)]  
= \langle h_1, h_2 \rangle_\H.
\] 
Let $\s\{W\}$ denote the $\s$-algebra generated by the noise $\dot{W}$.
That is, $\s\{W\}$ is the  $\s$-algebra generated by 
$\{ W(h): h\in C_c(\R_+\times\R^d)\}$.
The well-known Wiener-It\^o chaos decomposition
 (see, e.g., \cite[Theorem 1.1.2]{Nua06}) asserts that
$L^2(\O, \s\{W\}, \PP)$ can be decomposed into mutually 
orthogonal closed subspaces (called Wiener chaoses):

\noi
\begin{align}\label{alt1}
L^2(\O, \s\{W\}, \PP) 
= \bigoplus_{n=0}^\infty 
 \C^W_n,
\end{align}

\noi
where $ \C^W_0 \simeq \R$ is the set of constant random variables, 
$ \C^W_n$ is called $n$-th Wiener chaos that consists of 
all multiple integrals of order $n$; see, e.g., \cite[Section 2.7]{blue}.
The $n$-th multiple integral operator $I_n$ is a bounded linear operator
from the $n$-th tensor product $\H^{\otimes n}$ to $\C^W_n$ with
the following orthogonality relation

\noi
\begin{align}\label{OP1}
\E[ I_n(f) I_m(g) ] = \ind_{\{n=m\}} n! \jb{ \wt{f},  \wt{g}  }_{\H^{\otimes n}}
\end{align}

\noi
with $\wt{f}$ denoting the canonical symmetrization of $f$;
see, e.g., \cite[Appendix B]{blue} for the Hilbert space notation. 
Alternative to \eqref{alt1}, we can express the 
Wiener-It\^o chaos decomposition as follows:
for any $F\in L^2(\O, \s\{W\}, \PP)$, there exist
\emph{symmetric}
kernels 
$f_n\in\H^{\otimes n}$ such that 

\noi
\begin{align}\label{alt2}
F  = \E[F] + \sum_{n\geq 1} I_n(f_n).
\end{align}
The membership of $F$ in $L^2(\O)$ is equivalent to the finiteness of 
$\sum_{n\geq 1} n! \| f_n\|^2_{\H^{\otimes n}}$. 
With this chaos expansion, we can define several Malliavin operators
in a convenient manner.

\subsection{Basic Malliavin calculus.} \label{SEC2_M}

Let us first define several relevant Malliavin operators
using the above  chaos expansions \eqref{alt2}, and present results specifically tailored to the solutions to \eqref{pam} and \eqref{HAM}.

\medskip

\noi
$\bul$ {\bf Malliavin derivative operator}. For $k\in\{1,2\}$, we let $\DD^{k,2}$ denote
the set of all square integrable random variables $F$ with the Wiener-It\^o decomposition as in \eqref{alt2}, such that 
\[
\sum_{n\geq 1} n! n^k \| f_n\|^2_{\H^{\otimes n}} < \infty.
\] 
For any $F\in\DD^{2,2}$ with the Wiener-It\^o decomposition  as in \eqref{alt2}, the Malliavin derivative $D F$ and second order Mallivin derivative $D^2 F$ are given by
\[
D F =  \sum_{n\geq 1} n  I_{n-1}(f_n)
\quad{\rm and}
\quad
D^2 F =  \sum_{n\geq 2} n(n-1)  I_{n-2}(f_n),
\]
which are random vectors in $\H$ and 
$\H^{\otimes 2}$, respectively. 
Note that the Hilbert space $\H$ may contain generalized 
functions, so that the random `function' $DF$ may not be valued pointwise in general.  The same comment applies to $D^2F$,
and when $D^2F$ is indeed a function, 
$D_{s,y}D_{r,z}F$ is symmetric
 in $(s,y)$ and $(r,z)$ so that often 
  we   state   bounds  only for   $r < s$;
  see, e.g., \eqref{DU} below.
  When $u(t, x)$ denotes the  solution to \eqref{pam} or \eqref{HAM}
  in this note, 
we have the following results.

\begin{proposition} \label{prop_MD}
Under the assumptions of Theorem \ref{thm_main}, we have 
$u(t,x)\in\DD^{2, 4} \supset \DD^{2,2}$, meaning that 
$|u(t,x)| + \| D u(t,x) \|_{\H} +\| D^2 u(t,x) \|_{\H\otimes \H} \in L^4(\O)$. Moreover,
the map $(r,y)\in\R_+\times\R^d\mapsto D_{r, y} u(t,x) \in\R$
is indeed a {\rm(}random{\rm)} function such that 
 for any finite $p\geq 2$
and for almost every $0< r < s < t \leq T$, and $x, y\in\R^d$,

\noi
\begin{align} \label{DU}
\big\| D_{s,y}   u(t,x) \big\|_{p} \lesssim_T G_{t - s} (x - y)
\,\,\,
{\rm and}
\,\,\,
\big\| D_{r, z}D_{s, y}   u(t,x) \big\|_{p} \lesssim_T G_{t - s}(x - y) G_{s - r} (y - z)
,
 \end{align}

\noi
where
 the above implicit constants in $\les_T$ do not
depend on $(r, s, t)$ but depend on $T$;
 see \cite[Theorem 3.1]{NXZ22}
 and   \cite[Theorem 1.3]{BNQSZ} 
for more details.

\end{proposition}

For notational convenience, we say a random variable 
$F$ satisfies the property \eqref{ABSH} if 
$F\in\mathbb{D}^{2,2}$ and almost surely, 
$DF\in | \H|$ and
$D^2 F\in | \H^{\otimes 2}|$
meaning that 

\noi
\begin{align} \tag{\textbf{P}} \label{ABSH}
\begin{aligned}
 &\begin{cases}
\text{$(r,y)\in\R_+\times\R^d\mapsto | D_{r, y} F| \in\R$
belongs to $\H$}\\
\text{$(r,y, s, z)\in(\R_+\times\R^d)^2\mapsto | D_{s,z}D_{r, y} F| \in\R$
belongs to $\H^{\otimes2}$}.
\end{cases}
\end{aligned}
\end{align}

It is not difficult, via direct computations, to 
show that  the solution $u(t,x)$ in Proposition \ref{prop_MD}
and the corresponding spatial integrals 
 \eqref{def_f} satisfy the property \eqref{ABSH}.

\medskip

\noi
$\bul$ {\bf Ornstein-Uhlenbeck operators}. For $F\in\DD^{2,2}$  written as in \eqref{alt2},
we define 
$
LF = \sum_{n\geq 1} -n I_n(f_n).
$
Let $F\in L^2(\O)$ written as in  \eqref{alt2}. Suppose that  $\E[F] = 0$, 
we define 
$
L^{-1}F = \sum_{n\geq 1} -\frac{1}{n} I_n(f_n).
$
Here, $L$ is called the Ornstein-Uhlenbeck operator associated to 
the noise $\dot{W}$ and its domain coincides with $\DD^{2,2}$.
The operator $L^{-1}$ is called the pseudo-inverse of $L$ due to 
the fact that
$L^{-1} LF  = F - \E[ F]$ for $F\in\DD^{2,2}$,
and $LL^{-1}F = F$ for $F\in L^2(\O, \s\{W\}, \PP)$ with zero mean. 
We define $P_t = e^{t L}$ for $t\in\R_+$, which is called the Ornstein-Uhlenbeck
semigroup and satisfies the contraction property:

\noi
\begin{align}\label{conP}
\| P_t F \|_p \leq  \| F \|_p
\end{align}

\noi
for any $F\in L^p(\O, \s\{ W\}, \PP)$
with  $p\in[1,\infty)$. 
It is not difficulty to see
that for $F$ as in \eqref{alt2},
$P_t F = \E[F]  + \sum_{n\geq 1} e^{-tn} I_n(f_n)$.
Then, using the orthogonality relation \eqref{OP1}, we get 
$\| P_t F\|^2_2 = | \E[F] |^2  + \sum_{n\geq 1} e^{-2tn} \| I_n(f_n)\|^2_2
\leq \| F\|_2^2$ with equality when and only when $F$ is a constant 
or $t = 0$. This gives a  proof of \eqref{conP} for $p=2$.
The general case can be easily 
proved by using Mehler formula (see, e.g., \cite[Proposition 2.8.6]{blue}).

 Using the chaos expansion, it is not difficult to show that 
 $
 -D_\bul L^{-1}F = \int_0^\infty e^{-t} P_t D_\bul F dt
 $, 
 from which, with Minkowski's inequality and \eqref{conP},
  we can have

\noi
  \begin{align} \label{ineq1a}
    \begin{aligned} 
   \| D_\bul L^{-1}F\|_p 
   &\leq   \int_0^\infty e^{-t}  \| P_t D_\bul F \|_p dt  \\
&\leq    \int_0^\infty e^{-t}  \|  D_\bul F \|_p dt 
= \|  D_\bul F \|_p 
  \end{aligned}  
  \end{align}
  
  \noi
 for any $p\in[1,\infty)$, whenever $D_\bul F$ is a function.

 \medskip

\noi
$\bul$ {\bf Integration by parts formula and chain rule.}
The divergence operator $\dl$ is the adjoint operator for $D$, 
which can be characterized by the following integration by part
formula:
\begin{align}\label{ibp0}
\E[ \langle DF, u \rangle_\H ] =\E[ F \dl(u) ]
\end{align}

\noi
 for any $F\in\mathbb{D}^{1,2}$ and $u\in\text{dom}(\dl)$.
 Here $\text{dom}(\dl)$ is the set of random vector $u\in L^2(\Omega; \H)$
 such that 
 there is some finite constant $C_u$
 satisfying $ |\E[ \langle DF, u \rangle_\H ] | \leq C_u\| F\|_2$
 for any $F\in\mathbb{D}^{1,2}$.
 It is not difficult to prove via chaos expansion that
 $L = -\dl D$ on $\mathbb{D}^{2,2}$ (see, e.g., \cite[Proposition 1.4.3]{Nua06}).
 For $\phi:\R\to\R$ Lipschitz, differentiable and  $F\in\mathbb{D}^{1,2}$,
 it is known that $\phi(F)\in\mathbb{D}^{1,2}$ with
 $D\phi(F) = \phi'(F) DF $;
 see \cite[Proposition 1.2.3]{Nua06}. Then, we can easily derive
 the following formula:

\noi
\begin{align}\label{Stein_b2}
 \E[ G\phi(F) ] 
 =  \E \big[ \langle DF, - DL^{-1}G \rangle_\H \phi'(F) \big]
\end{align}
for any differentiable and Lipschitz function $\phi$, $F\in\mathbb{D}^{1,2}$, and  
$G\in L^2(\Omega)$ with $\E[G]=0$.
Indeed, using $G = L L^{-1}G$ and the above chain rule with 
\eqref{ibp0} and $L = -\dl D$, we get 

\noi
\begin{align*}
 \E[ G\phi(F) ] 
 &=  \E[ -\dl D L^{-1}G\phi(F) ] 
 =\E[ \langle - D L^{-1}G, D\phi(F) \rangle_\H ]  \\
 &= \E[ \langle - D L^{-1}G, DF  \rangle_\H \phi'(F)  ].
\end{align*}
Note that taking $\phi(x)=x$, we get 

\noi
\noi
\begin{align}\label{Stein_b2bb}
 {\rm Cov}(F, G)
 =  \E \big[ \langle DF, - DL^{-1}G \rangle_\H \big]
\end{align}
for centered random variable $G$  with finite second moment
and $F\in\mathbb{D}^{1,2}$.

Now  we  state a key bound in this note,
which arises in the so-called improved second-order 
Gaussian Poincar\'e inequality;
see \cite{Vid20, BNQSZ}.

\begin{proposition}\textup{(\cite[Proposition 1.9]{BNQSZ})}
\label{lmm_d-dl}
Recall the notation
in \eqref{ABSH}.
 Let $F_1, F_2$ be centered random variables in $\DD^{2,4}$ such that
 $D F_j \in |\H|$ and $D^2 F_j\in  |\H^{\otimes 2}|$ almost surely
 for $j=1,2$, {\rm(}i.e., $F_1, F_2$ satisfy the property  \eqref{ABSH}{\rm)}.
  Then, 
 
 \noi
 \begin{align*}
  \Var \big(  \jb{DF_1, -DL^{-1} F_2 }_\H \big) \lesssim \cA (F_1, F_2) + \cA (F_2, F_1),
 \end{align*}
 where  
  \begin{align} \notag 
 \begin{aligned}
\cA(F_1, F_2) &
\coloneqq 
 \int_{\R_+^6\times\R^{6d}}  dr dr' dsds' d\theta d\theta'  dzdz' dydy' d\w d\w' \\
&\qquad \times 
\g_0(s-s') \g_0(r-r') \g_0(\theta - \theta')  
\g(z-z') \g(y-y') \g(\w - \w')  \\
&\qquad  \times \| D_{r, z} D_{\theta, \w} F_1 \|_4  
\| D_{s,y}D_{\theta', \w'} F_1 \|_4  
\| D_{r', z'} F_2\|_4  \| D_{s', y'} F_2\|_4. 
 \end{aligned}
 \end{align}
 
 \end{proposition}

\medskip

\subsection{Proof of  Proposition \ref{prop1818}}
\label{SEC2_IL}
 
We conclude this section with the proof of  Proposition \ref{prop1818}.

\begin{proof}[Proof of Proposition \ref{prop1818}]

 According to Proposition \ref{prop:IL},
 the ASCLT holds  for  $\big\{ F_\theta: \theta\geq 1 \big\}$    if 

 \noi
\begin{align}
    \sup_{|s| \leq T}
    \int_2^\infty
    \dfrac{\E \big[  |\K_t(s)|^2\big]  }{t  \log t} dt <\infty
\notag
\end{align}
for any finite $T >0$, where  $\K_t(s)$ is   as in \eqref{Kn}.
 Expanding $|\K_t(s)|^2$, 
we get

\noi
\begin{align*}
 |\K_t(s)|^2
 &=  \frac{1}{(\log t)^2}
 \int_{[1,t]^2} \frac{1}{\theta\w}
\big( e^{is   F_\theta} - e^{- \frac{s^2}{2} } \big)
\big( e^{- is   F_\w} - e^{- \frac{s^2}{2} } \big) d\theta d\w  \\
 &=  \frac{1}{(\log t)^2}
 \int_{[1,t]^2} \frac{1}{\theta\w}
       \big( e^{is (  F_\theta -   F_\w)} + e^{-s^2} - e^{is   F_\theta } 
       e^{- \frac{s^2}{2} }    - e^{ - is   F_\w } e^{-\frac{s^2}2 } \big) d\theta d\w  \\
 &=  \mathbb{I}_t(s) -  e^{ - \frac{s^2}{2}}   \II_t(s),
 \end{align*}

 \noi
 where

 \noi
 \begin{align}\label{12s}
 \begin{aligned}
 \mathbb{I}_t(s)
 &\coloneqq
  \frac{1}{(\log t)^2}
 \int_{[1,t]^2} \frac{1}{\theta\w}\Big( e^{i s (  F_\theta -  F_\w)} -  e^{-s^2} \Big)\, 
 d\theta d\w ,
\\
  \II_t(s) &\coloneqq  \frac{1}{\log t} \int_1^t \frac{1}{\theta} \big( e^{i s   F_\theta } 
  + e^{- i s   F_\theta } - 2 e^{-\frac{s^2}{2}} \big) \, d\theta .
 \end{aligned}
 \end{align}

 \noi
Therefore, it suffices to show that
\begin{align}
\label{A12}
\mathbf A_1(s) \coloneqq \int_2^\infty \frac{  \E\big[ \mathbb{I}_t(s) \big] }{t\log t} dt
 \quad
  \text{and}
   \quad
   \mathbf A_2(s) \coloneqq \int_2^\infty \frac{  \E\big[ \II_t(s) \big] }{t\log t} dt,
    \quad s\in [-T, T]
\end{align}

\noi
are both uniformly bounded for any given $T>0$.

\medskip

\noi
$\bul$ {\bf Estimation for $\mathbf{A}_2$.~} 
Recall that
  $\E [e^{i s Y}] = e^{- \frac{s^2}{2}}$ with $Y\sim \NN(0,1)$,
  and for any real random variable $X$,
\begin{align} \label{WA0}
\big| \E [e^{i s X}]  -\E [e^{i s Y}]  \big| 
= \big| \E [e^{i s X}] - e^{- \frac{s^2}{2}} \big| \leq 4 \, d_{\rm TV}(X, Y),
\end{align}

\noi
where the total-variation distance $ d_{\rm TV}$  is defined as in
\eqref{TVD}. While using the local Lipschitz property
of the complex exponentials, 
we have 
\begin{align}\label{WA1}
\sup_{|s|\leq T}\big| \E [e^{i s X}]  -\E [e^{i s Y}]  \big| 
\leq  2T \, d_{\rm Wass}(X, Y).
\end{align}

Therefore, it follows from \eqref{WA0}-\eqref{WA1} that

\noi
\begin{align}
\begin{aligned}
\Big| \mathbb{E}\big[ e^{i s   F_\theta } 
+ e^{- i s    F_\theta} - 2 e^{-\frac{s^2}{2}} \big]  \Big|
&  \leq \min\{  8 \, d_{\rm TV}\big(  F_\theta, Y\big), 
4T \, d_{\rm Wass}(F_\theta, Y) \}.
\end{aligned}
\label{T4}
\end{align}

\noindent

Therefore, it follows 
from \eqref{12s}, \eqref{T4}, and \eqref{TV1a} that  
for any finite $T > 0$,

\noi
\begin{align*}
\sup \big\{ |\mathbf A_2(s)| : s\in [-T, T]   \big\}
\les \int_2^\infty \frac{1}{t \log^2 t} \int_1^t \frac{1}{\theta^{1+\beta_1}} d\theta dt < \infty,
\end{align*}

\noi
that is,  $\sup\{ |\mathbf A_2(s)| : s\in [-T, T]   \} <\infty$
for any finite $T>0$.

\medskip

\noi
$\bul$ {\bf Estimation for $\mathbf{A}_1$.~}  Using  the inequality,
 \begin{align*}
\Big| \E\big[ e^{i s (  F_\theta -   F_\w)} -  e^{-s^2}  \big] \Big|
&= \Big| \E \Big[ e^{i \sqrt{2} s \big(\frac{  F_\theta -   F_\w}{\sqrt{2}}\big)} 
             - e^{i \sqrt{2}s Y}  \Big] \Big| \\
&\leq   \min\Big\{ 4\,  d_{\rm TV}\Big( \tfrac{  F_\theta -   F_\w}{\sqrt{2}}, Y \Big),
2\sqrt{2}T \,  d_{\rm Wass}\Big( \tfrac{  F_\theta -   F_\w}{\sqrt{2}}, Y \Big) \Big\},
\end{align*}
we can write, with ${\rm dist} = d_{\rm TV}$ or $d_{\rm Wass}$, 

 \noi
 \begin{align}
 \sup_{|s|\leq T} \mathbf A_1(s)
& \les   \int_2^\infty \frac{1 }{t (\log t)^3}
 \bigg( \int_{[1,t]^2} \frac{1}{\theta\w} \,
 {\rm dist}\big( \tfrac{  F_\theta -   F_\w}{\sqrt{2}}, Y \big)
 \, d\theta d\w   \bigg)  dt  \notag \\
 &\les  \int_2^\infty \frac{1 }{t (\log t)^3}
 \bigg( \int_{1 < \theta < \w < t} \frac{1}{\theta\w} \,
  {\rm dist}\big( \tfrac{  F_\theta -   F_\w}{\sqrt{2}}, Y \big)
 \, d\theta d\w  \bigg)  dt ,
  \label{T5a}
 \end{align}

\noi
which is finite due to the assumption \eqref{TV1b}.

Hence the proof of Proposition \ref{prop1818} is completed.
\qedhere

\end{proof}

\section{Proof of Theorem \ref{thm_main}}\label{sec_proof}

In this section, we provide the proof of our main result Theorem \ref{thm_main}.  
According to Proposition \ref{prop1818},
we  only need to check the conditions \eqref{TV1a}
and  \eqref{TV1b}
for  the spatial integrals $\wh F_R$ \eqref{def_f} in place of $F_\theta$.

Note that 
Theorem \ref{thm_old} implies that
in any of the four cases in \eqref{QCLT1}, there is some positive constant
$ b > 0$ such that 

\noi
\begin{align}\label{QCLT2}
d_{\rm TV}(\wh F_\theta, Y) \leq C \theta^{-b}.
\end{align}

\noi
That is, the condition  \eqref{TV1a} is verified. 

Next, we will show that    there exist positive real numbers 
$\beta_1$ and  $\beta_2$
such that

\noi
\begin{align}\label{claima}
d_{\rm TV}\Big( \tfrac{\wh F_\theta - \wh F_\w}{\sqrt{2}}, Y \Big)
\les  \theta^{-\beta_1}   + (\theta/\w)^{\beta_2}
\end{align}

\noi
for   $1 < \theta < \w <\infty$.
This constitutes   the bulk of the proof.

The bound in \eqref{claima} is obtained using techniques analogous to those employed in deriving the bound in \eqref{QCLT2}. We provide a brief outline of these techniques below.
Observe that the random variable $ \tfrac{\wh F_\theta - \wh F_\w}{\sqrt{2}}$
has mean zero and variance 

\noi
\begin{align}\label{VVAR}
V_{\theta, \w} \coloneqq {\rm Var}\Big(\tfrac{\wh F_\theta - \wh F_\w}{\sqrt{2}}\Big)
= 1 - {\rm Cov} \big(\wh F_\theta, \wh F_\w \big).
\end{align}

\noi
Thus,  applying Stein's bound (see, e.g., \cite[Theorem 3.3.1]{blue}),
we get 

\noi
\begin{align}\label{Stein_bdd}
d_{\rm TV}( G, Y)
\leq \sup  \big| \E[ G\phi(G)  - \phi'(G)] \big|  
\quad\text{with $G\coloneqq\tfrac{\wh F_\theta - \wh F_\w}{\sqrt{2}}$},
\end{align}

\noi
where the above supremum runs over bounded, differentiable functions
$\phi:\R\to\R$ with $\|\phi\|_\infty\leq \sqrt{\pi/2}$ and $\| \phi' \|_\infty\leq 2$.
In view of \eqref{Stein_b2bb},
 the inner product $\langle DG, - DL^{-1}G \rangle_\H$
 has mean 
\[
\E[ \langle DG, - DL^{-1}G \rangle_\H ] = \E[ G^2] = V_{\theta, \w}.
\]
Then, it follows from \eqref{Stein_bdd}, \eqref{VVAR},
and \eqref{Stein_b2} 
with the Cauchy-Schwarz  inequality  that 

\noi
\begin{align}
\label{Stein_b3}
\begin{aligned}
d_{\rm TV}\Big(  \tfrac{\wh F_\theta - \wh F_\w}{\sqrt{2}}, Y \Big)
& \leq 2  \big| 1 - V_{\theta, \w}\big|  
+  2  \E\big|  \langle DG, - DL^{-1}G \rangle_\H  - V_{\theta, \w}\big|  \\
&\leq  2 \big|  {\rm Cov}(\wh F_\theta, \wh F_\w  ) \big|
 +  \sqrt{ \Var \big(  \big\langle D (\wh F_\theta - \wh F_\w), 
 -DL^{-1} (\wh F_\theta - \wh F_\w) \big\rangle_\H \big) }.
\end{aligned}
\end{align}

\noi
Recalling our goal \eqref{claima}, it suffices to show that, 
for any $1< \theta < \w <\infty$,

 \noi
 \begin{align} \label{COV_b1}
 \big|  {\rm Cov}(\wh F_\theta, \wh F_\w  ) \big|
 \les  (\theta/\w)^{\be_2}
\end{align}
with   $\be_2 = \tfrac{d}{2}$ in
\textbf{(case 1)} and \textbf{(case 3)} 
and 
 $\be_2 = \tfrac{\al}{2}$ in
\textbf{(case 2)} and \textbf{(case 4)}
specified in \eqref{QCLT1},
where the above implicit constant does not depend on 
$(\theta, \w)$. 
The bound \eqref{COV_b1} will be proved in  Section \ref{SEC22}.
For the variance term in \eqref{Stein_b3}, 
it is essential to estimate 
\begin{align}\label{term1}
\Var \big(  \big\langle D \wh F_\theta, 
 - DL^{-1}  \wh F_\w \big\rangle_\H \big)
 \end{align}
 
 \noi
for $\theta, \w\in(1,\infty)$, concerning the bilinearity of the inner product operation,
 the linearity of the operators $D$ and $L^{-1}$,
 and the elementary inequality
 $\Var(X_1+X_2) \leq 2\Var(X_1) +2\Var(X_2) $
 for any square-integrable random variables $X_1, X_2$.
When $\theta = \w$, the estimate for \eqref{term1} has been established in 
\cite[Section 3.1]{NXZ22} and 
\cite[Section 4.2]{BNQSZ}, where it is shown that, with an implicit constant independent of $\theta$,

\noi
\begin{align}\label{ob1}
\Var \big(  \big\langle D \wh F_\theta, 
 -DL^{-1}  \wh F_\theta \big\rangle_\H \big)
 \les \begin{cases}
   \theta^{-d} \quad\text{in \textbf{(case 1)} and \textbf{\textbf{(case 3)}}} \\[0.5em]
    \theta^{-\al} \quad\text{in \textbf{(case 2)} and \textbf{(case 4)}}.
 \end{cases}
\end{align}

\noi
The derivation  of \eqref{ob1} relies on the ideas around
the so-called  second-order Gaussian Poincar\'e inequality
(\cite{Ch08, Vid20}),  see Proposition \ref{lmm_d-dl}. This inequality, utilized in \cite{NXZ22, BNQSZ}, will also
 play a crucial role when estimating the term \eqref{term1}
for $\theta \neq \w$. In Section \ref{SEC23},
we will show for $1 < \theta < \w$:

\noi
\begin{align}  \label{ob2}
\begin{aligned} 
& \Var \big(  \big\langle D \wh F_\theta, 
 -DL^{-1}  \wh F_\w \big\rangle_\H \big)
 +\Var \big(  \big\langle D \wh F_\w, 
 -DL^{-1}  \wh F_\theta \big\rangle_\H \big) \\
 &\qquad \qquad
 \les \begin{cases}
   \theta^{-d} \quad\text{in \textbf{(case 1)} and \textbf{\textbf{(case 3)}}} \\[0.5em]
    \theta^{-\al} \quad\text{in \textbf{(case 2)} and \textbf{(case 4)}}.
 \end{cases}
\end{aligned}
\end{align}
Therefore, the claim \eqref{claima}
follows immediately from 
\eqref{Stein_b3}, \eqref{COV_b1},
\eqref{ob1}, and \eqref{ob2}.
Hence the proof is complete. 
\qed

\bigskip

It remains 
to 
justify \eqref{COV_b1} and \eqref{ob2}.

\subsection{
Proof of (\ref{COV_b1})} \label{SEC22}

 In this subsection, we will show \eqref{COV_b1} for 
 the four cases specified in \eqref{QCLT1}.
First, we present the precise asymptotic relation of the limiting variance of $F_R$ in \eqref{def_f}, as cited from \cite[Theorems 1.6 and 1.7]{NZ20a}
for \eqref{pam} and 
\cite[Theorem 1.4]{BNQSZ}
for \eqref{HAM}:

 \noi
 \begin{align}\label{sR_asymp}
 \s_R = \sqrt{\Var (F_R)} \sim 
 \begin{cases}
  R^{\tfrac{d}{2}}  &\quad\text{in \textbf{(case 1)} and \textbf{\textbf{(case 3)}}}\\[0.5em]
   R^{d - \tfrac\al2} &\quad\text{in \textbf{(case 2)} and \textbf{(case 4)}},
 \end{cases}
 \end{align}

 \noi
 where we write $a_R\sim b_R$ to mean
 $0< \liminf_{R\to+\infty }a_R/ b_R \leq  \limsup_{R\to+\infty }a_R/ b_R < +\infty$.
 

\medskip
\noi
$\bul$ {\bf  \textbf{(case 1)} and \textbf{\textbf{(case 3)}}.}
  It follows from \cite[Theorem 1.6 for PAM]{NZ20a} for \textbf{(case 1)}
  and \cite[Formulas (4.8) and (4.10) for HAM]{BNQSZ} for \textbf{\textbf{(case 3)}}
  that 
  there exists a \textbf{non-negative} function $\Phi_{s,t} \in L^1 (\R^d)$ such that
  
  \noi
  \begin{align}\label{def_Phi}
    \Phi_{s,t} (x - y) \coloneqq {\rm Cov}\big(  u(t,x), u(s,y) \big).
  \end{align}
   As a result, taking \eqref{sR_asymp} into  account, 
   we get for  $R' \geq R > 1$, 

\noi
\begin{align*}
\big| {\rm Cov}( \wh{F}_R ,  \wh{F}_{R'}   ) \big|
&\les (RR' )^{-\frac{d}{2} } \bigg\vert 
\int_{|x| < R} dx\int_{|y|  <R'} dy \, {\rm Cov}\big(  u(t,x), u(t,y) \big)  \bigg\vert \\
&\leq (RR' )^{- \frac{d}{2} }
\int_{|x| < R} dx\int_{\R^d} dz \Phi_{t,t} (z) \les  \big( \tfrac{R}{R'} \big)^{d/2}.
\end{align*} 
This 
verifies \eqref{COV_b1} for \textbf{(case 1)}
and \textbf{\textbf{(case 3)}}.

\medskip

Note that in \textbf{(case 2)} and \textbf{(case 4)}, 
the function $\Phi_{s,t}$  in \eqref{def_Phi}
does not belong to $L^1(\R^d)$.
As a result, we need to carry out another approach
 to settle this difficulty.

 \medskip

 \noi
 $\bul$ {\bf \textbf{(case 2)} and \textbf{(case 4)}.} Recall that $\wh F_R$ is a centered 
 random variable in $\DD^{1,2}$. Then, using 
 an integration-by-part formula from \cite[Theorem 2.9.1]{blue},
 we can  first write 
 
 \noi
 \begin{align} \notag
 \E[ \wh F_R \wh F_{R'} ]
 = \E\big[  \langle D\wh F_R, - DL^{-1} \wh F_{R'} \rangle_\H \big]
 \end{align}
 (see also \eqref{Stein_b2})
 and then apply the definition \eqref{iso_inner} and \eqref{ineq1a} 
 with Fubini's theorem and 
  Cauchy-Schwarz inequality to get
 
 \noi
 \begin{align} \label{ibp2}
  \begin{aligned} 
& \big|   \E[ \wh F_R \wh F_{R'} ]  \big| \\
 & \leq 
\E \bigg[\int_{\R_+^2\times\R^{2d}}
   | D_{r, y}\wh F_R |  \times  | - D_{r', y'}L^{-1} \wh F_{R'}| \g_0(r-r') \g_1(y-y')
   drdr' dydy' \bigg] \\
   &\leq 
    \int_{\R_+^2\times\R^{2d}} 
     \big\| D_{r, y}\wh F_R \big\|_2  \times   \big\| D_{r', y'}\wh F_{R'} \big\|_2 \g_0(r-r') \g_1(y-y')
   drdr' dydy'. 
 \end{aligned} 
 \end{align}
 Using the basic property of Malliavin derivative operator
 and the bound in \eqref{DU},
 we have 
 
 \noi
\begin{align}  \label{ibp3} 
\begin{aligned}  
\big\| D_{r, y}\wh F_R \big\|_2 
&\leq \frac{1}{\s_R} \int_{|x| < R}  \| D_{r, y} u(t_0, x)  \|_2 dx \\
&\les   \frac{1}{\s_R} \int_{|x| < R} G_{t_0 -r}(x-y)  dx.
 \end{aligned} \end{align}
 
 Therefore, combining \eqref{ibp2} with \eqref{ibp3} leads us to 
 
 \noi
 \begin{align} \label{ibp4}
  \begin{aligned} 
\big|  \E[ \wh F_R \wh F_{R'} ] \big|
& \les 
\frac{1}{\s_R \s_{R'} }   \int_{\R_+^2\times\R^{2d}} 
  \int_{|x| <R}   \int_{ |x'| < R'}   G_{t_0 -r}(x-y)  G_{t_0 -r'}(x'-y')   \\
  & \quad \times \g_0(r-r') |y-y' |^{-\al} drdr' dydy' dx dx',
 \end{aligned} 
 \end{align}
 
 \noi
 where we recall from \eqref{sR_asymp} that $\s_R\sim R^{d - \tfrac\al2}$
 and from \eqref{QCLT1} that $\g_1(z) = |z|^{-\al}$ for 
  $\al\in(0, 2\wedge d)$.
  
In the following, we 
will use the Fourier analysis to get fine estimates of the above spatial integral
\eqref{ibp4}. Let us fix some notations: for an integrable function
$g:\R^d\to\R$, its Fourier transform $\wh g$ is given by 
$\wh{g}(\xi) = \int_{\R^d} e^{-ix\cdot \xi} g(x) dx$.
Recall the expressions of wave/heat kernels 
in \eqref{fSol} and 
we record below their Fourier transforms:

\noi
\begin{align} \label{fSolF}
\wh{G^H_t}(\xi)  = e^{-\frac{t}{2} |\xi|^2}
\quad
{\rm and}
\quad
\wh{G^W_t}(\xi) = \tfrac{\sin(t|\xi| )}{|\xi|}, \quad \text{for all } t > 0, \xi \in \R^d.
\end{align}

\noi
Then, 
using
Plancherel's theorem, we get 
from \eqref{ibp4} with \eqref{sR_asymp} that 

\noi
\begin{align}
   \big| \E[ \wh{F}_R   \wh{F}_{R'} ] \big| \lesssim 
   &  (RR' )^{\frac{\alpha}{2} - d} \int_0^{t_0} \int_0^{t_0} drdr' \g_0(r-r') \int_{\R^d} d\xi 
    \notag  \\
   &\quad \times
\int_{|x| < R} dx  \int_{|x'|<R'} dx'   e^{-i\xi\cdot(x-x') }
   \wh{G}_{t_0-r}(\xi)    \wh{G}_{t_0-r'}(\xi) | \xi| ^{\al - d}    \notag    \\
   = & 
    (\tfrac{R'}{R} )^{\frac\alpha2}  \int_0^{t_0} \int_0^{t_0} drdr' \g_0(r-r')\int_{\R^d} d\xi     
    \notag  \\
   &\quad \times
\int_{|x| < 1}   dx  \int_{|x'|<1} dx'    e^{-i\xi\cdot \big(x- \tfrac{R'}{R}x'\big) }
   \wh{G}_{t_0-r}(\xi /R)    \wh{G}_{t_0-r'}(\xi/R) | \xi| ^{\al - d},   \label{Fa1}
\end{align}
 
\noi
where in the last equality,  we performed 
a change of variable $(x, x', \xi) \mapsto (Rx, R'x', \xi/R)$.

Next, we further bound \eqref{Fa1} in  \textbf{(case 2)} and \textbf{(case 4)} separately. 
For \textbf{(case 2)}, recalling \eqref{fSolF}, 
and using 

\noi
\begin{align} \notag 
  |\xi|^{- 2 \be} = \frac{1}{ \Gamma(\be)} \int_0^{\infty} d \theta 
  e^{-\theta |\xi|^2} \theta^{\be - 1}
\end{align}

\noi
with $\be =\tfrac{d-\al}{2} >0$,
and making a  change of variables
$(t_0 - r, t_0 - r') \mapsto (r, r')$, 
we can  get from \eqref{Fa1} 
that

\noi
\begin{align}
\label{Fa2}
\begin{aligned}
 \big| \E[ \wh{F}_R   \wh{F}_{R'} ] \big|   
  \lesssim 
  & (\tfrac{R}{R'})^{\frac\al2}  \int_0^{t_0} \int_0^{t_0} drdr' \g_0(r-r')
    \int_{|x|, |x'| < 1} dx dx'  \int_0^{\infty} d\theta  \theta^{\frac{d - \alpha  - 2}{2} }\\
  &\qquad \times \int_{\R^d} d\xi e^{- i (x - \tfrac{R'}{R}x') \cdot \xi -  (\theta + \frac{r + r'}{2R^2}) |\xi|^2} \\
  = & (2\pi)^{\frac{d}{2}} (\tfrac{R}{R'})^{\frac\al2}  \int_0^{t_0} \int_0^{t_0} drdr' \g_0(r-r')
    \int_{|x|, |x'| < 1} dx dx'  \\
  &\qquad \times     \int_0^{\infty} d\theta  \theta^{\frac{d - \alpha  - 2}{2} }   
   \big( \theta +  \tfrac{r + r'}{2R^2} \big)^{-\frac d2}
    \exp\Big( -\tfrac{|x - \frac{R'}{R} x'|^2}{4 ( \theta + \frac{r + r'}{2R^2}) }   \Big),
\end{aligned}
\end{align}

\noi
where in the last step, we use the Fourier transform of the heat kernel.

To deal with the  integration  in $\theta$ from \eqref{Fa2}, we decompose
 the  region $(0, \infty)$ into two segments $(0,\frac{r + r'}{2R^2})$ 
 and $(\frac{r + r'}{2R^2}, \infty)$. 
Since $r,r' \in (0,t)$, it is easy to deduce that

\noi
\begin{align}
\begin{aligned}
& \int_0^{\frac{r + r'}{2R^2}} d\theta  \theta^{\frac{d - \alpha  - 2}{2} }   
   \big( \theta +  \tfrac{r + r'}{2R^2} \big)^{-\frac d2}
    \exp\Big( -\tfrac{|x - \frac{R'}{R} x'|^2}{4 ( \theta + \frac{r + r'}{2R^2}) }   \Big) \\
\les
 &  \big(\tfrac{R^2}{r + r'}\big)^{\frac{d}{2}}   
 \exp\Big(-\tfrac{|x - \tfrac{R'}{R} x'   |^2}{4(r + r')/R^2}\Big) 
 \int_0^{\frac{r + r'}{2R^2}} d\theta \theta^{\frac{d - \alpha - 2}2}    
 \qquad\text{with $\tfrac{d - \alpha - 2}2 > -1$} \\
 \les
 & \big(\tfrac{R^2}{r + r'}\big)^{\frac{d}{2}}   
 \exp\Big(-\tfrac{|x - \tfrac{R'}{R} x'   |^2}{4(r + r')/R^2}\Big) 
\big(  \tfrac{r + r'}{2R^2} \big)^{\frac{d-\al}{2}}   \\
=
&  \big(\tfrac{R^2}{r + r'}\big)^{\frac{\al}{2}}   
 \exp\Big(-\tfrac{|x - \tfrac{R'}{R} x'   |^2}{4(r + r')/R^2}\Big) 
\les |x - \tfrac{R'}{R} x'   |^{-\al},
\end{aligned}
\label{Fa3a}
\end{align}

\noi
and

\noi
\begin{align}
\begin{aligned}
 \int_{\frac{r + r'}{2R^2}}^\infty d\theta  \theta^{\frac{d - \alpha  - 2}{2} }   
   \big( \theta +  \tfrac{r + r'}{2R^2} \big)^{-\frac d2}
   & \exp\Big( -\tfrac{|x - \frac{R'}{R} x'|^2}{4 ( \theta + \frac{r + r'}{2R^2}) }   \Big) 
 \leq     \int_{\frac{r + r'}{2R^2}}^\infty d\theta  \theta^{\frac{ - \alpha  - 2}{2} }   
    e^{ -\tfrac{|x - \frac{R'}{R} x'|^2}{8 \theta   }  } \\
\leq &   \int_{0}^\infty d\theta  \theta^{\frac{ - \alpha  - 2}{2} }   
    e^{ -\tfrac{|x - \frac{R'}{R} x'|^2}{8 \theta   }  }
\les    |x - \tfrac{R'}{R} x'   |^{-\al}.
\end{aligned}
\label{Fa3b}
\end{align}

\noi Therefore, we deduce from \eqref{Fa2}, \eqref{Fa3a}, and \eqref{Fa3b}
with the local integrability of $\g_0$
that

\noi
\begin{align}
\label{Fa3c}
\begin{aligned}
 \big| \E[ \wh{F}_R   \wh{F}_{R'} ] \big|   
& \les
 (\tfrac{R}{R'})^{\frac\al2}  \int_0^{t_0} \int_0^{t_0} drdr' \g_0(r-r')
    \int_{|x|, |x'| < 1} dx dx'   |x - \tfrac{R'}{R} x'   |^{-\al} \\
   &=
    (\tfrac{R}{R'})^{-\frac\al2}  
    \int_{|x| < 1} dx    \int_{|x'| < 1}   dx'   |\tfrac{R}{R'} x -  x'   |^{-\al}  
    \les   (\tfrac{R}{R'})^{-\frac\al2},  
\end{aligned}
\end{align}

\noi
where for the last step, we used the following elementary fact that
\begin{equation}\label{Fa3d}
\sup_{z\in\R^d}\int_{|x| < 1}  |z -   x  |^{-\be}  dx  < \infty, \quad\forall \be\in(0, d).
\end{equation}

\noi
 Thus, the proof of \eqref{COV_b1} in \textbf{(case 2)} is complete. 

\medskip

Now let us deal with \textbf{(case 4)}. For the wave kernel in $\R^d$
with $d\leq 2$, it enjoys the following 
property that with $\ind_R(x) \coloneqq  \ind_{\{ |x | < R\}}$,
\begin{align} \label{speed}
(\ind_R \ast G_r)(y) 
\coloneqq \int_{\R^d}  \ind_{\{ |x| < R\}} G_r(x-y) dx 
\leq  r \ind_{R+r}(y),
\end{align}

\noi
which follows from the definition \eqref{fSol} of the wave kernel;
see also \cite[Lemma 2.1]{BNZ}.
Note that the Fourier transform of  $\ind_R$ is a 
real-valued and rotationally symmetric function (see, e.g., \cite[Lemma 2.1]{NZ20a}).
Then, using Plancherel's theorem and some Fourier calculations, 
we derive from \eqref{ibp4} with \eqref{sR_asymp} and
a change of variables $(t_0 - r, t_0- r')\to (r, r')$
that 

\noi
 \begin{align*}  
&\big|  \E[ \wh F_R \wh F_{R'} ] \big|
 \les 
 (RR')^{\frac{\al}{2} -d }  \int_0^{t_0} \int_0^{t_0} dr dr' \g_0(r-r')  \int_{\R^d} d\xi  |\xi |^{\al-d} 
   \wh{G}_r(\xi)  \wh{G}_{r'}(\xi)   \\
 & \qquad\qquad\qquad \quad \times \bigg(   \int_{ |x| <R }    dx  \int_{ |x'| < R'} dx' e^{-i(x-x') \cdot \xi} \bigg)  \\
 &= (RR')^{\frac{\al}{2} -d }  \int_0^{t_0} \int_0^{t_0} dr dr' \g_0(r-r')  
 \int_{\R^d} d\xi  |\xi |^{\al-d} 
 \wh{\ind_R}(\xi)  \wh{\ind_{R'}}(\xi)   \wh{G}_r(\xi)  \wh{G}_{r'}(\xi)   \\
 &= C_{\al, d} (RR')^{\frac{\al}{2} -d }  \int_0^{t_0} \int_0^{t_0} dr dr' \g_0(r-r')  
 \int_{\R^{2d}}  (\ind_R\ast G_r)(y)  (\ind_{R'}\ast G_{r'})(y') |y-y'|^{-\al}dydy' , 
 \end{align*}
 
 \noi
 where the constant $C_{\al, d}$ comes from inverting the 
 Fourier transform of the spectral measure $|\xi |^{\al-d}d\xi$.
 Now applying the inequality \eqref{speed} with $r, r' \in (0, t_0)$
 and utilizing the local integrability of $\g_0$,
 we get 
 
 \noi
 \begin{align*}  
\big|  \E[ \wh F_R \wh F_{R'} ] \big|
 &\les 
  t_0^2  (RR')^{\frac{\al}{2} -d }  \int_0^{t_0} \int_0^{t_0} dr dr' \g_0(r-r')  
 \int_{\R^{2d}}  \ind_{R+t_0}(y)   \ind_{R' +t_0}(y') |y-y'|^{-\al}dydy' \\
 &\les 
 (RR')^{\frac{\al}{2} -d }  \int_{\R^{2d}}  \ind_{R+t_0}(y)   \ind_{R' +t_0}(y') |y-y'|^{-\al}dydy' \\
 &\les  (RR')^{\frac{\al}{2} -d } (R+t_0)^d  (R'+t_0)^{d-\al}
 \les (\tfrac{R}{R'})^{\frac{\al}{2}},
 \end{align*}

\noi
where we obtained the last second step with the same argument as in 
\eqref{Fa3c}-\eqref{Fa3d}.
Hence, the proof of \eqref{COV_b1} in \textbf{(case 4)} is finished. 
\qed

\subsection{Proof of (\ref{ob2})} \label{SEC23}

 To prove \eqref{ob2}, we first apply 
 Proposition \ref{lmm_d-dl} with 
 the   bounds in \eqref{DU}
 and \eqref{ibp3}:
 with $B_R : = \{ x\in\R^d: |x| < R\}$,
 we need to bound 
 
\noi
\begin{align*}
  \mathcal{A} \big(\widehat{F}_R, \widehat{F}_{R'} \big) 
  \les   
  & \s_R^{-2} \s_{R'}^{-2} \int_0^{t_0} dr \int_0^{r} d \theta 
  \int_0^{t_0} dr' \int_0^{r'} d \theta'  \int_{[0,t]^2} ds ds' 
  \int_{\R^{6d}} dzdz' dydy' d\w d\w' \\
&\qquad \times 
\g_0(s-s') \g_0(r-r') \g_0(\theta - \theta')  \g(z-z') \g(y-y') \g(\w - \w')  \\
&\qquad \times \int_{B_R^2} d x_1 d x_2 
G_{t - r} (x_1 - z) G_{r - \theta} (z - \w) 
G_{t - r'} (x_2 - z') G_{r' - \theta'} (z' - \w')   \\
&\qquad \times \int_{B_{R'}^2} d x_3 d x_4 G_{t - r'}(x_3 - y) G_{t - s'} (x_4 - y').
\end{align*}

In \textbf{(case 1)}, we simply enlarge the region $B_R$ to $B_{R'}$. 
Then, applying the estimates for $\mathcal{A}^*$ 
appearing in \cite[Section 3.1.1]{NXZ22} and 
recalling from \eqref{sR_asymp} that  $\s_R=\sigma_R(t_0) \sim R^{\frac{d}{2}}$,
one can easily conclude that
\begin{align*}
  \mathcal{A} \big(\widehat{F}_R, \widehat{F}_{R'} \big) 
  \les 
   R^{-d}.
  \end{align*}

For \textbf{(case 2)}, we follow the idea employed in \cite[Section 3.1.2]{NXZ22}.
That is, for i.i.d. standard normal random variables $Z_1,\dots, Z_6$,
we have

\noi
\begin{align*}
  \mathcal{K} \coloneqq 
  & \int_{B_1^4} dx_1 \cdots d x_4
   \E\bigg[  \Big|  \frac{\sqrt{t-r}}{R} Z_1 - \frac{\sqrt{r-\theta}}{R}  Z_2 
   - \frac{\sqrt{t-s}}{R}     Z_3    +  \frac{\sqrt{s-\theta'}}{R}Z_4 + x_1-x_2   \Big|^{-\al} \\
  &\quad \quad \times  \Big| \frac{\sqrt{t-s}}{R}    Z_3
     -   \frac{\sqrt{t-s'}}{R'} Z_6 +x_2-x_4 \Big|^{-\alpha}  
  \times  \Big|  \frac{\sqrt{t-r}}{R} Z_1 - \frac{\sqrt{t-r'}}{R'}Z_5 + x_1 - x_3  \Big|^{-\al} \bigg]\\
  \les & 1,
\end{align*}

\noi
and thus,  the expression
\begin{align*}
  S_R = & S_R (t,r,r',s,s',\theta,\theta')\\
   \coloneqq & \int_{B_R^2} d x_1 d x_2 \int_{B_{R'}^2} d x_3 d x_4 \int_{\R^{6d}} dzdz' dydy' d\w d\w' 
 \g(z-z') \g(y-y') \g(\w - \w')    \\
&\quad  \times  G_{t - r} (x_1 - z)  G_{r - \theta} (z - \w) G_{t - r'} (x_2 - z') G_{r' - \theta'} (z' - \w') G_{t - r'}(x_3 - y) G_{t - s'} (x_4 - y') \\
= & R^{2d - 2\alpha} (R')^{2d - \alpha} \mathcal{K} 
\les  R^{2d - 2\alpha} (R')^{2d - \alpha}.
\end{align*}

\noi
Then, with $\sigma_R \sim R^{d - \frac\alpha2}$, we get  

\noi
\begin{align*}
  \mathcal{A} \big(\widehat{F}_R, \widehat{F}_{R'} \big) 
  \les  \sigma_R^{-2} \sigma_{R'}^{-2} \int_0^{t_0} dr \int_0^{r} d \theta
   \int_0^{t_0} dr' \int_0^{r'} d \theta'  \int_{[0,t]^2} ds ds' S_R \lesssim R^{-\alpha}.
\end{align*}

Turning to the cases for \eqref{HAM}, we first consider 
  \textbf{(case 3)}.  Just like   \textbf{(case 1)}, 
  with extending the integrating region in $x_1, x_2$ to $B_{R'}$, 
  applying the estimates for $\mathcal{A}_R$ in \cite[Page 809]{BNQSZ}, 
  we have
$  \mathcal{A} \big(\widehat{F}_R, \widehat{F}_{R'} \big) \lesssim R^{-d}$;
and for   \textbf{(case 4)},  following a similar argument as in \cite[Section 4.2.2]{BNQSZ}, one can deduce that
$
  \mathcal{A} \big(\widehat{F}_R, \widehat{F}_{R'} \big) \lesssim R^{-\alpha}.
$
The estimates for $\mathcal{A} (\widehat{F}_{R'}, \widehat{F}_{R} )$ 
are also very similar and thus omitted here. Therefore, we have

\noi
\begin{align*}
\mathcal{A} \big(\widehat{F}_{R}, \widehat{F}_{R'} \big) 
+ \mathcal{A} \big(\widehat{F}_{R'}, \widehat{F}_{R} \big) \lesssim \begin{dcases}
R^{-d}, & \text{\textbf{(case 1)} and \textbf{(case 3)} }\\
(R')^{-\alpha} < R^{-\alpha}, & \text{\textbf{(case 2)} and    \textbf{(case 4)}      }.
\end{dcases}
\end{align*}

\noi
The proof of \eqref{ob2} is then  complete by invoking  Proposition \ref{lmm_d-dl}. 
\qed

\quad\\
\noi
{\bf $\bul$ Acknowledgement.} The authors are grateful to the referee and the editor 
for the suggestions and comments.

\quad\\
\noi
{\bf Data Availability Declaration:} No data is used.  
 
\quad\\
\noi
{\bf Author Contribution Declaration}: The authors have made equal contribution in preparing this work.

\quad\\
\noi
{\bf Competing Interest Declaration}: The authors declare that
 the authors have no competing interests as defined by Springer, 
 or other interests that might be perceived to influence the results and/or discussion reported in this paper.

\end{document}